\newcount\ite\ite=1\def\0{\global\ite=1\1}
\def\1{\item{\rm(\romannumeral\the\ite)}\advance\ite1\quad}

\documentclass[12pt]{amsart}

\pdfoutput=1

\usepackage{epic, eepic, amsfonts, latexsym, amssymb, graphicx,
multicol, mathrsfs, color, amscd, verbatim, paralist,
xspace, url, euscript, stmaryrd,  amsmath, enumitem,
bbold, multirow, tikz}
\usepackage[all]{xy}

\usepackage[colorlinks, linkcolor=blue, citecolor=magenta, urlcolor=cyan]{hyperref}

\font\teneufm=eufm10 scaled \magstep1
\font\seveneufm=eufm7 scaled \magstep1
\font\fiveeufm=eufm5  scaled \magstep1
\newfam\eufmfam

\textfont\eufmfam=\teneufm
\scriptfont\eufmfam=\seveneufm
\scriptscriptfont\eufmfam=\fiveeufm

\newfam\msbfam
\font\tenmsb=msbm10 scaled \magstep1  \textfont\msbfam=\tenmsb
\font\sevenmsb=msbm7 scaled \magstep1 \scriptfont\msbfam=\sevenmsb
\font\fivemsb=msbm5 scaled \magstep1  \scriptscriptfont\msbfam=\fivemsb

\def\dd#1{\raise1.5pt\hbox{$\,\partial\!$}/\raise-2.5pt\hbox{$\!\partial#1\,$}}

\def\tilde{\widetilde}
\def\hat{\widehat}
\def\5#1{{\mathcal #1}}

    \newcommand\cQ{\mathcal{Q}} 

\def\fm{{\mathfrak m}}

\def\CC{{\mathbb C}}

\def\NN{{\mathbb N}}

\def\PP{{\mathbb P}}

\def\XX{{\mathbb X}}

\def\ra{\rightarrow}

\def\hat{\widehat}
\def\reg{\mathop{\rm reg}\nolimits}

\def\GL{\mathop{\rm GL}\nolimits}
\def\SL{\mathop{\rm SL}\nolimits}

\def\Ann{\mathop{\rm Ann}\nolimits}

\def\s{\mathop{\rm s}\nolimits}
\def\ss{\mathop{\rm ss}\nolimits}

\def\Soc{\mathop{\rm Soc}\nolimits}

\def\Gr{\mathop{\rm Gr}\nolimits}
\def\Sym{\mathop{\rm Sym}\nolimits}

\def\gitq{/\hspace{-0.1cm}/}

\newcommand\co{\colon} %macro to use in f\co \rightarrow Y
 \def\HollowBoxx #1#2#3{{\dimen0=#1 \advance\dimen0 by -#2
       \dimen1=#1 \advance\dimen1 by #3
        \vrule height 0pt depth #3 width #2
       \hskip -#3
       \vrule height #1 depth #3 width #3}}
 \def\LeftContraction{\mathord{\kern1.45pt \HollowBoxx{6pt}{3.5pt}{.4pt}}\,}

 \def\HollowBox #1#2#3{{\dimen0=#1 \advance\dimen0 by -#3
       \dimen1=#1 \advance\dimen1 by #3
        \vrule height #1 depth #3 width #3
        \vrule height 0pt depth #3 width #2
        \hskip -#3}}
 \def\RightContraction{\mathord{\, \HollowBox{6pt}{3.1pt}{.4pt}} \kern1.6pt}

\newtheorem{theorem}{THEOREM}[section]

\newtheorem{proposition}[theorem]{Proposition}
\newtheorem{conjecture}[theorem]{Conjecture}

\theoremstyle{definition}

\newtheorem{lemma}[theorem]{Lemma}

\theoremstyle{remark}
\newtheorem{remark}[theorem]{Remark}

\makeatletter
\def\blfootnote{\xdef\@thefnmark{}\@footnotetext}
\makeatother

\begin{document}

\title[Associated forms of binary quartics and ternary cubics]{Associated forms of binary quartics
\vspace{0.1cm}\\ 
and ternary cubics}
\author[Alper]{J. Alper}
\author[Isaev]{A. V. Isaev}
\author[Kruzhilin]{N. G. Kruzhilin}

\address[Alper]{Mathematical Sciences Institute\\
Australian National University\\
Canberra, ACT 0200, Australia}
\email{jarod.alper@anu.edu.au}

\address[Isaev]{Mathematical Sciences Institute\\
Australian National University\\
Canberra, ACT 0200, Australia}
\email{alexander.isaev@anu.edu.au}

\address[Kruzhilin]{Department of Complex Analysis\\
Steklov Mathematical Institute\\
8 Gubkina St., Moscow GSP-1 119991, Russia}
\email{kruzhil@mi.ras.ru}

\maketitle

\thispagestyle{empty}

\pagestyle{myheadings}

\begin{abstract}

Let $\cQ_n^d$ be the vector space of forms of degree\linebreak $d\ge 3$ on $\CC^n$, with $n\ge 2$. The object of our study is the map $\Phi$, introduced in \cite{EI}, \cite{AI1}, that assigns every nondegenerate form in $\cQ_n^d$ the so-called associated form, which is an element of $\cQ_n^{n(d-2)*}$. We focus on two cases: those of binary quartics ($n=2$, $d=4$) and ternary cubics ($n=3$, $d=3$). 
In these situations the map $\Phi$ induces a rational equivariant involution on the projectivized space $\PP(\cQ_n^d)$, which is in fact the only nontrivial rational equivariant involution on $\PP(\cQ_n^d)$. In particular, there exists an equivariant involution on the space of elliptic curves with nonvanishing $j$-invariant. In the present paper, we give a simple interpretation of this involution in terms of projective duality. Furthermore, we express it via classical contravariants.
\end{abstract}

\section{Introduction}\label{intro}
\setcounter{equation}{0}

In this paper we continue to explore new ideas in classical invariant theory that were proposed in the recent article \cite{EI} and further developed in \cite{AI1}, \cite{AI2}. Let ${\mathcal Q}_n^d := \Sym^d (\CC^{n*})$ be the vector space of forms of degree $d$ on $\CC^n$, where $n\ge 2$, $d\ge 3$. Assuming that the discriminant of $f\in{\mathcal Q}_n^d$ does not vanish, define $M_f:=\CC[z_1,\dots,z_n]/(f_{z_1},\dots,f_{z_n})$ to be the Milnor algebra of the isolated hypersurface singularity at the origin of the zero set of $f$. Let $\fm$ be the maximal ideal of $M_f$. One can then introduce a form defined on the $n$-dimensional quotient $\fm/\fm^2$ with values in the one-dimensional socle $\Soc(M_f)$ of $M_f$ as follows:
$$
\begin{aligned}
\fm/\fm^2 	& \to \Soc(M_f), \\
x & \mapsto y^{\,n(d-2)},
\end{aligned}
$$
where $y$ is any element of ${\mathfrak m}$ that projects to $x\in{\mathfrak m}/{\mathfrak m}^2$.  There is a canonical isomorphism ${\mathfrak m}/{\mathfrak m}^2\cong \CC^{n*}$ and, since the Hessian of $f$ generates the socle, there is also a canonical isomorphism $\Soc(M_f) \cong \CC$. Hence, one obtains a form ${\mathbf f}$ of degree $n(d-2)$ on  $\CC^{n*}$ (i.e.~an element of $\Sym^{n(d-2)}(\CC^n)$), which is called the {\it associated form}\, of $f$ (see Section \ref{setup} for more detail on this definition).

It is a consequence of Corollary 3.3 in \cite{AI1} that, upon identification of $\CC^{n*}$ with $\CC^n$, the associated form of $f$ is a Macaulay inverse system of the Milnor algebra $M_f$. Furthermore, if we identify the space $\Sym^{n(d-2)} (\CC^n)$ with $(\Sym^{n(d-2)} (\CC^{n*}))^*=\cQ_n^{n(d-2)*}$ by means of the polar pairing, then the form ${\mathbf f}$ coincides, up to scale, with the element of ${\mathcal Q}_n^{n(d-2)*}$ given by
$$
\begin{aligned}
{\mathcal Q}_n^{n(d-2)}  	& \to \CC, \\
g					& \mapsto \frac{1}{(2\pi i)^n}\int\limits_{|f_{z_{{}_1}}|=\varepsilon_1,\dots,|f_{z_{{}_n}}|=\varepsilon_n}\frac{g\, dz_1\wedge\cdots\wedge dz_n}{f_{z_1}\cdots f_{z_n}}.
\end{aligned}
$$
For a discussion of these equivalent ways to describe the associated form we refer the reader to \cite{AI2}.

The principal object of our study is the morphism
$$
\Phi:X_n^d\to \cQ_n^{n(d-2)*},\quad f\mapsto{\mathbf f}
$$
of affine algebraic varieties, where $X_n^d$ is the variety of forms in $\cQ_n^d$ with nonzero discriminant. This map has a $\GL_n$-equivariance property, and one of the reasons for our interest in $\Phi$ is the following intriguing conjecture proposed in \cite{AI1} (see also \cite{EI}):

\begin{conjecture}\label{conj2} For every regular $\GL_n$-invariant function $S$ on $X_n^d$ there exists a rational $\GL_n$-invariant function $R$ on $\cQ_n^{n(d-2)*}$ defined at all points of the set\, $\Phi(X_n^d)\subset {\mathcal Q}_n^{n(d-2)*}$ such that $R\circ\Phi=S$.
\end{conjecture}

If confirmed, the conjecture would imply that the invariant theory of forms in ${\mathcal Q}_n^d$ can be extracted, by way of the morphism $\Phi$, from that of forms in ${\mathcal Q}_n^{n(d-2)*}$ at least at the level of rational invariant functions, or absolute invariants. In \cite{EI}, Conjecture \ref{conj2} was shown to hold for binary forms (i.e.~for $n=2$) of degrees $3\le d\le 6$, and in \cite{AI1} its weaker variant was established for arbitrary $n$ and $d$. Furthermore, in \cite{AI2} the conjecture was confirmed for binary forms of any degree. While Conjecture \ref{conj2} is rather interesting from the purely invariant-theoretic viewpoint, it has an important implication for singularity theory. Namely, as explained in detail in \cite{AI1}, \cite{AI2}, if this conjecture is established, it will provide a solution, in the homogeneous case, to the so-called {\it reconstruction problem}, which is the question of finding a constructive proof of the well-known Mather-Yau theorem (see \cite{MY}, \cite{Sh}). Settling Conjecture \ref{conj2} is part of our program to solve the reconstruction problem for quasihomogeneous isolated hypersurface singularities. This amounts to showing that a certain system of invariants introduced in \cite{EI} is complete, and Conjecture \ref{conj2} implies completeness in the homogeneous case.

The morphism $\Phi$ is rather natural and deserves attention regardless of Conjecture \ref{conj2}. In fact, this map is interesting even for small values of $n$ and $d$. In the present paper, we study $\Phi$ in two situations: those of binary quartics ($n=2$, $d=4$) and ternary cubics ($n=3$, $d=3$). These are the only choices of $n$, $d$ for which $\Phi$ preserves the form's degree. Curiously, as we will see in Section \ref{result}, in each of the two cases the projectivization ${\mathbb \Phi}$ of $\Phi$ induces an equivariant involution on the image $\XX_n^d$ of $X_n^d$ in the projective space $\PP(\cQ_n^d)$, with one $\SL_n(\CC)$-orbit removed. Furthermore, as we show in Theorem \ref{invclass}, a nontrivial rational equivariant involution on $\PP(\cQ_2^4)$ and $\PP(\cQ_3^3)$ is unique. In particular, ${\mathbb \Phi}$ yields a unique equivariant involution on the space of elliptic curves with nonvanishing $j$-invariant, which appears to have never been mentioned in the extensive literature on elliptic curves. Early observations in this direction go back to article \cite{Ea} published some 10 years ago but so far the involution has not been understood in more explicit terms. 

The main goals of the present paper are twofold. Firstly, for binary quartics and ternary cubics we describe the equivariant involution via projective duality. Namely, in Section \ref{result} we prove that for $f \in \XX_n^d$ the element ${\mathbb \Phi}(f) \in \PP(\cQ_n^{d*})=\PP(\cQ_n^d)^*$ is identified with the tangent space of the $\GL_n(\CC)$-orbit of $\hat f$ at $\hat f$, where $\hat f$ is any lift of $f$ to $\cQ_n^d$ (see Theorem \ref{main}). Secondly, in Section \ref{S:contravariant} we consider the contravariant defined by $\Phi$ and relate it to classical contravariants due to Cayley and Sylvester, which gives yet another interpretation of the equivariant involution induced by ${\mathbb \Phi}$. This section is written in the spirit of mid-19th century invariant theory with focus on explicit formulas and identities. 

{\bf Acknowledgements.}  This work is supported by the Australian Research Council. It was initiated during the third author's visit to the Australian National University, and significant progress was made during the second author's stay at the Max Planck Institute for Mathematics in Bonn in 2014.

\section{Preliminaries}\label{setup}
\setcounter{equation}{0}
 
Let ${\mathcal Q}_n^d$ be the vector space of forms of degree $d$ on $\CC^n$ where $n\ge 2$. Its dimension is given by the well-known formula
\begin{equation}
\dim_{\CC}{\mathcal Q}_n^d=\left(
\begin{array}{c}
d+n-1\\
d
\end{array}
\right).\label{dimform}
\end{equation}
The standard action of $\GL_n=\GL_n(\CC)$ on $\CC^n$ induces an action on ${\mathcal Q}_n^d$ as follows:
$$
(C\cdot f)(z):=f\left(C^{-1}\cdot z\right)
$$
for $C\in\GL_n$, $f\in{\mathcal Q}_n^d$ and $z=(z_1,\dots,z_n)\in\CC^n$. Two forms that lie in the same $\GL_n$-orbit are called linearly equivalent. Below we will be mostly concerned with the induced action of $\SL_n=\SL_n(\CC)$.  

To every nonzero $f\in{\mathcal Q}_n^d$ we associate the hypersurface
$$
V_f:=\{z\in\CC^n:f(z)=0\}
$$
and consider it as a complex space with the structure sheaf induced by $f$. The singular set of $V_f$ is then the critical set of $f$. In particular, if $d\ge 2$ the hypersurface $V_f$ has a singularity at the origin. We are interested in the situation when this singularity is isolated, or, equivalently, when $V_f$ is smooth away from 0. This occurs if and only if $f$ is nondegenerate, i.e.~$\Delta(f)\ne 0$, where $\Delta$ is the discriminant (see Chapter 13 in \cite{GKZ}). 

For $d\ge 3$ define
$$
X^d_n:=\{f\in{\mathcal Q}_n^d:\Delta(f)\ne 0\}.
$$
Observe that $\GL_n$ acts on the affine variety $X_n^d$ and note that every $f\in X_n^d$ is stable with respect to this action, i.e.~the orbit of $f$ is closed in $X_n^d$ and has dimension $n^2$  (see, e.g.,~Corollary 5.24 in \cite{Mu}).

Fix $f\in X^d_n$ and consider the Milnor algebra of the singularity\ of $V_f$, which is the complex local algebra
$$
M_f:=\CC[[z_1,\dots,z_n]]/(f_1,\dots,f_n),
$$
where $\CC[[z_1,\dots,z_n]]$ is the algebra of formal power series in $z_1,\dots,z_n$ with complex coefficients and $f_j:=\partial f/\partial z_j$, $j=1,\dots,n$. Since the singularity of $V_f$ is isolated, the algebra $M_f$ is Artinian, i.e.\linebreak $\dim_{\CC}M_f<\infty$ (see Proposition 1.70 in \cite{GLS}). Therefore, $f_1,\dots,f_n$ is a system of parameters in $\CC[[z_1,\dots,z_n]]$. Since $\CC[[z_1,\dots,z_n]]$ is a regular local ring, $f_1,\dots,f_n$ is a regular sequence in $\CC[[z_1,\dots,z_n]]$. This yields that $M_f$ is a complete intersection. 

It is convenient to utilize another realization of the Milnor algebra. Namely, we can write
$$
M_f=\CC[z_1,\dots,z_n]/(f_1,\dots,f_n).
$$
Let ${\mathfrak m}$ denote the maximal ideal of $M_f$, which consists of all elements represented by polynomials in $\CC[z_1,\dots,z_n]$ vanishing at the origin. The maximal ideal is nilpotent and we let $\nu:=\max\{\eta\in\NN\mid {\mathfrak m}^{\eta}\ne 0\}$ be the socle degree of $M_f$.

Since $M_f$ is a complete intersection, by \cite{B} it is a Gorenstein algebra. This means that the socle of $M_f$, defined as
$$
\Soc(M_f):=\{x\in{\mathfrak m}: x\,{\mathfrak m}=0\},
$$
is a one-dimensional vector space over $\CC$  (see, e.g.,~Theorem 5.3 in \cite{Hu}). We then have $\Soc(M_f)={\mathfrak m}^{\nu}$. Furthermore, $\Soc(M_f)$ is spanned by the element of $M_f$ represented by the Hessian $H(f)$ of $f$ (see, e.g.,~Lemma 3.3 in \cite{Sa}). Since $H(f)$ is a form in ${\mathcal Q}_n^{n(d-2)}$, it follows that $\nu=n(d-2)$ (see \cite{AI1}, \cite{AI2} for details). Thus, the subspace 
\begin{equation}
W_f:={\mathcal Q}_n^{n(d-2)-d+1}f_1+\dots+{\mathcal Q}_n^{n(d-2)-d+1}f_n\subset{\mathcal Q}_n^{n(d-2)}\label{subspace}
\end{equation}
has codimension 1, with the line spanned by $H(f)$ being complementary to it.

Let $e_1^*,\dots,e_n^*$ be the basis in $\CC^{n*}$ dual to the standard basis in $\CC^n$ and $z_1^*,\dots,z_n^*$ the coordinates of a vector $z^*\in\CC^{n*}$ (we slightly abuse notation by writing $z^*=(z_1^*,\dots,z_n^*)$). Denote by $\omega \co \Soc(M_f)\ra\CC$ the linear isomorphism given by the condition $\omega(H(f))=1$ (with $H(f)$ viewed as an element of $M_f$). Define a form ${\mathbf f}$ of degree $n(d-2)$ on $\CC^{n*}$ (i.e.~an element of $\Sym^{n(d-2)} (\CC^n)$) by the formula
$$
{\mathbf f}(z^*):=\omega\left((z_1^*{\mathbf z}_1+\dots+z_n^*{\mathbf z}_n)^{n(d-2)}\right),\label{assocformdef}
$$
where ${\mathbf z}_j$ is the element of the algebra $M_f$ represented by the coordinate function $z_j\in\CC[z_1,\dots,z_n]$. We call ${\mathbf f}$ the {\it associated form}\, of $f$. 

The associated form arises from the following map:
\begin{equation}
\begin{array}{rll}
\fm/\fm^2&\to&\Soc(M_f),\\
\vspace{-0.3cm}\\
x&\mapsto& y^{n(d-2)},
\end{array}\label{coordinatefree}
\end{equation}
with $y\in\fm$ being any element that projects to $x\in\fm/\fm^2$. Indeed, ${\mathbf f}$ is derived from this map by identifying the target with $\CC$ via $\omega$ and the source with $\CC^{n*}$ by mapping the image of ${\mathbf z}_j$ in $\fm/\fm^2$ to $e_j^*$, $j=1,\dots,n$.

To obtain an expanded expression for ${\mathbf f}$, notice that if  $i_1,\dots,i_n$ are nonnegative integers such that $i_1+\dots+i_n=n(d-2)$, the product ${\mathbf z}_1^{i_1}\cdots {\mathbf z}_n^{i_n}$ lies in $\Soc(M_f)$, hence we have 
\begin{equation}
{\mathbf z}_1^{i_1}\cdots {\mathbf z}_n^{i_n}=\mu_{i_1,\dots,i_n}(f) H(f)\label{assocformexpppp}
\end{equation}
for some $\mu_{i_1,\dots,i_n}(f)\in\CC$. In terms of the coefficients $\mu_{i_1,\dots,i_n}(f)$ the form ${\mathbf f}$ is written as 
\begin{equation}
{\mathbf f}(z^*)=\sum_{i_1+\cdots+i_n=n(d-2)}\frac{(n(d-2))!}{i_1!\cdots i_n!}\mu_{i_1,\dots,i_n}(f)
z_1^{* i_1}\cdots z_n^{* i_n}.\label{assocformexpp}
\end{equation}
Notice that each $\mu_{i_1,\dots,i_n}$ is a regular function on $X_n^d$, therefore 
\begin{equation}
\mu_{i_1,\dots,i_n}=\frac{P_{i_1,\dots,i_n}}{\Delta^{p_{i_1,\dots,i_n}}},\label{formulaformus}
\end{equation}
for some $P_{i_1,\dots,i_n}\in\CC[\cQ_n^d]$ and nonnegative integer $p_{i_1,\dots,i_n}$.

Recall that the polar pairing yields a canonical identification between the spaces $\Sym^{n(d-2)} (\CC^n)$ and $(\Sym^{n(d-2)} (\CC^{n*}))^*={\mathcal Q}_n^{n(d-2)*}$. Using this identification, we can regard the associated form as an element of ${\mathcal Q}_n^{n(d-2)*}$ and consider the morphism
$$
\Phi \co X_n^d\ra {\mathcal Q}_n^{n(d-2)*} ,\quad f\mapsto {\mathbf f}
$$
of affine varieties. This map is rather natural; in particular, Proposition 2.1 in \cite{AI1} implies an equivariance property for $\Phi$. Namely, introducing an action of $\GL_n$ on the dual space ${\mathcal Q}_n^{n(d-2)*}$ in the usual way as
$$
(C\cdot g)(h):=g(C^{-1}\cdot h),\quad g\in \cQ_n^{n(d-2)*},\, h\in {\mathcal Q}_n^{n(d-2)},\,C\in\GL_n,
$$
we have:
  
\begin{proposition}\label{equivariance} For every $f\in X_n^d$ and $C\in\GL_n$ the following holds:
$$
\Phi(C\cdot f)=(\det C)^2\,\Bigl(C\cdot\Phi(f)\Bigr).
$$
In particular, the morphism $\Phi$ is $\SL_n$-equivariant. 
\end{proposition}

\begin{remark}\label{olddef} In \cite{Ea}, \cite{AI1}, \cite{AI2} the associated form was defined as the element of $\cQ_n^{n(d-2)}=\Sym^{n(d-2)} (\CC^{n*})$ obtained from map (\ref{coordinatefree}) by identifying the quotient $\fm/\fm^2$ with $\CC^n$ rather than $\CC^{n*}$. Accordingly, the morphism $\Phi$ was introduced as a map from $X_n^d$ to $\cQ_n^{n(d-2)}$. The morphism so defined has the following equivariance property:
\begin{equation}
\Phi(C\cdot f)=(\det C)^2\,\Bigl((C^{-1})^T\cdot\Phi(f)\Bigr),\quad f\in X_n^d,\, C\in\GL_n.\label{equivarianceee}
\end{equation}
Below it will be sometimes convenient to view associated forms and the morphism $\Phi$ in this way. 
\end{remark}

The present paper mainly concerns two situations: the case of binary quartics and ternary cubics. In the next section, we will give a geometric description of the morphism $\Phi$ in terms of projective duality and in Section \ref{S:contravariant} an algebraic interpretation of $\Phi$ in terms of classical contravariants.

\section{Duality for binary quartics and ternary cubics}\label{result}
\setcounter{equation}{0}

We will now projectivize the setup of Section \ref{setup} and replace the action of $\GL_n$ with that of $\SL_n$. Namely, let $\PP(\cQ_n^d)$ be the projectivization of $\cQ_n^d$, i.e.~$\PP(\cQ_n^d): = (\cQ_n^d \setminus \{0\}) / \CC^{\times}$. In what follows we often write elements of $\PP(\cQ_n^d)$ as forms meaning that they are considered up to scale. The action of $\SL_n$ on $\cQ_n^d$ induces an $\SL_n$-action on $\PP(\cQ_n^d)$, and for $f\in \PP(\cQ_n^d)$ we denote its orbit $\SL_n\cdot f$ by $O(f)$. Further, define $\XX_n^d \subset \PP(\cQ_n^d)$ to be the image of $X_n^d$ under the quotient morphism $\cQ_n^d \setminus \{0\} \to \PP(\cQ_n^d)$. Clearly, for $f\in\XX_n^d$ the orbit $O(f)$ is closed in $\XX_n^d$ and has dimension $n^2-1$. Similarly, we projectivize the space $\cQ_n^{n(d-2)*}$ and consider the induced action of $\SL_n$ on $\PP(\cQ_n^{n(d-2)*})$. 

The map $\Phi$ descends to a morphism
$$
{\mathbb \Phi}\co \XX_n^d \to \PP(\cQ_n^{n(d-2)*}).
$$
By Proposition \ref{equivariance}, the morphism ${\mathbb \Phi}$ is equivariant:
$$
{\mathbb \Phi}(C\cdot f)=C\cdot {\mathbb \Phi}(f),\quad f\in\XX_n^d,\,\,C\in\SL_n. \label{equivarnewphi}
$$
Hence, in the case when ${\mathbb \Phi}$ maps the variety $\XX_n^d$ into the semistable locus $\PP(\cQ_n^{n(d-2)*})^{\ss}$ of $\PP(\cQ_n^{n(d-2)*})$, it gives rise to a morphism $\phi$ of good GIT quotients for which the following diagram commutes:
$$
\xymatrix{
\XX_n^d \ar[r]^{\hspace{-1cm}{\mathbb \Phi}} \ar[d] &\PP(\cQ_n^{n(d-2)*})^{\ss} \ar[d]\\
\XX_n^d\gitq\SL_n	\ar[r]^{\hspace{-1cm}\phi} & \PP(\cQ_n^{n(d-2)*})^{\ss}\gitq\SL_n.
}
$$
In the diagram, the quotient on the left is affine and geometric, and the one on the right is projective. Furthermore, $\XX_n^d$ is a Zariski open subset of the stable locus $\PP(\cQ_n^d)^{\s}$, hence the affine quotient $\XX_n^d\to\XX_n^d\gitq\SL_n$ is a restriction of the projective quotient $\PP(\cQ_n^d)^{\ss}\to\PP(\cQ_n^d)^{\ss}\gitq\SL_n$. Observe that the situation $n=2$, $d=3$ is trivial and can be excluded from consideration. Indeed, since all nondegenerate binary cubics are pairwise linearly equivalent, $\XX_2^3=\PP(\cQ_2^3)^{\ss}=\PP(\cQ_2^3)^{\s}$ is a single orbit and $\XX_2^3\gitq\SL_2$ is a point. For elementary introductions to GIT quotients and various notions of stability we refer the reader to \cite{Mu} and Chapter 9 in \cite{LR}.

We focus on the morphism ${\mathbb \Phi}$ in two cases. Indeed, notice that for all pairs $n,d$ (excluding the trivial situation $n=2$, $d=3$) one has $n(d-2)\ge d$, and the equality holds precisely for the following two pairs: $n=2$, $d=4$ and $n=3$, $d=3$. We will explain below that in  each of these two cases ${\mathbb \Phi}$ maps $\XX_n^d$ to $\PP(\cQ_n^{d*})^{\ss}$ and induces an equivariant involution on the variety $\XX_n^d$ with one orbit removed. Furthermore, we will see that such an involution is unique. For these purposes, in Subsections \ref{S:binary-quartics}--\ref{uniqueness} it will be convenient to regard associated forms as elements of $\cQ_n^d$ and ${\mathbb \Phi}$ as a map from $\XX_n^d$ to $\PP(\cQ_n^d)$ (see Remark \ref{olddef} for details).

Let us now describe the maps ${\mathbb \Phi}$ and $\phi$ in each of the two cases. Some of the facts that follow can be extracted from articles \cite{Ea}, \cite{EI}. 

\subsection{Binary quartics} \label{S:binary-quartics}
Let $n=2$, $d=4$. It is a classical result that every nondegenerate binary quartic is linearly equivalent to a quartic of the form
\begin{equation}
q_t(z_1,z_2):=z_1^4+tz_1^2z_2^2+z_2^4,\quad t\ne\pm 2\label{qt}
\end{equation}
(see pp.~277--279 in \cite{El}). A straightforward calculation yields that the associated form of $q_t$ is
\begin{equation}
{\mathbf q}_t(z_1,z_2):=\frac{1}{72(t^2-4)}(tz_1^4-12z_1^2z_2^2+tz_2^4).\label{bfqt}
\end{equation}
For $t\ne 0,\pm 6$ the quartic ${\mathbf q}_t$ is nondegenerate, and in this case the associated form of ${\mathbf q}_t$ is proportional to $q_t$, hence $\mathbb{\Phi}^2(q_t) = q_t$. As explained below, the exceptional quartics $q_0$, $q_6$, $q_{-6}$, are pairwise linearly equivalent.

It is easy to show that $\PP(\cQ_2^4)^{\ss}$ is the union of $\XX_2^4$ (which coincides with $\PP(\cQ_2^4)^{\s}$) and two orbits that consist of strictly semistable forms:\linebreak $O_1:=O(z_1^2z_2^2)$, $O_2:=O(z_1^2(z_1^2+z_2^2))$, of dimensions 2 and 3, respectively. Notice that $O_1$ is closed in $\PP(\cQ_2^4)^{\ss}$ and is contained in the closure of $O_2$. We then observe that ${\mathbb \Phi}$ maps $\XX_2^4$ onto $\PP(\cQ_2^4)^{\ss}\setminus (O_2\cup O_3)$, where $O_3:=O(q_0)$ (as we will see shortly, $O_3$  contains the other exceptional quartics $q_6$, $q_{-6}$ as well). Also, notice that ${\mathbb \Phi}$ maps the 3-dimensional orbit $O_3$ onto the 2-dimensional orbit $O_1$ (thus the stabilizer of $q_0$ is finite while that of $\mathbb{\Phi}(q_0)$ is one-dimensional). In particular, ${\mathbb \Phi}$ restricts to an equivariant involutive automorphism of $\XX_2^4\setminus O_3$, which for $t\ne 0,\pm 6$ establishes a duality between the quartics $C\cdot q_t$ and $(C^{-1})^T\cdot q_{-12/t}$ with $C\in\SL_2$, hence between the orbits $O(q_t)$ and $O(q_{-12/t})$ (see (\ref{equivarianceee})).  

In order to understand the induced map $\phi$ of GIT quotients, we note that the algebra of invariants $\CC[\cQ_2^4]^{\SL_2}$ is generated by a pair of elements $I_2$, $I_3$ (the latter invariant is called the catalecticant),
 where the subscripts indicate their degrees (see, e.g.,~pp.~41, 101--102 in \cite{El}). One has 
\begin{equation}
\Delta=I_2^3-27\,I_3^2,\label{deltabinquar}
\end{equation}
and for a binary quartic of the form
$$
f(z_1,z_2)=az_1^4+6bz_1^2z_2^2+cz_2^4
$$
the values of $I_2$ and $I_3$ are computed as 
\begin{equation}
\begin{array}{l}
I_2(f)=ac+3b^2,\quad I_3(f)=abc-b^3.\label{form1}
\end{array}
\end{equation}
It then follows that the algebra $\CC[X_2^4]^{\GL_2}\simeq\CC[\XX_2^4]^{\SL_2}$ is generated by the invariant
\begin{equation}
J:=\frac{I_2^3}{\Delta}.\label{form2}
\end{equation}
Therefore, the quotient $X_2^4\gitq\GL_2\simeq \XX_2^4\gitq\SL_2$ is the affine space $\CC$, and $\PP(\cQ_2^4)^{\ss}\gitq\SL_2$ can be identified with $\PP^1$, where both $O_1$ and $O_2$ project to the point at infinity in $\PP^1$. 

Next, from formulas \eqref{qt}, (\ref{deltabinquar}), \eqref{form1}, \eqref{form2} we calculate 
\begin{equation}
J(q_t)=\frac{(t^2+12)^3}{108(t^2-4)^2}\quad\hbox{for all $t\ne \pm 2$.}\label{form3}
\end{equation}
Clearly, \eqref{form3} yields
\begin{equation}
J(q_0)=J(q_6)=J(q_{-6})=1,\label{Jeq1}
\end{equation}
which implies that $q_0$, $q_6$, $q_{-6}$ are indeed pairwise linearly equivalent as claimed above and that the orbit $O_3$ is described by the condition $J=1$. 

Using (\ref{bfqt}), \eqref{form3} one obtains
\begin{equation}
J({\mathbf q}_t)=\frac{J(q_t)}{J(q_t)-1}\quad\hbox{for all $t\ne 0,\pm 6$.}\label{jtransfbinquar}
\end{equation}
Furthermore, the calculations leading to (\ref{jtransfbinquar}) also yield the following identities for any $f\in X_2^4$:
\begin{equation} \label{binary-pullbacks}
I_2(\mathbf{f}) = \frac{I_2(f)}{2^8 3^3 \Delta(f)}, \,\,
I_3(\mathbf{f}) = -\frac{1}{2^{12} 3^6 \Delta(f)}, \,\, \Delta(\mathbf{f}) = \frac{I_3(f)^2}{2^{24} 3^6 \Delta(f)^3}.
\end{equation}
Hence, we observe: $I_3(\mathbf{f}) \neq 0$ (that is, the catalecticant of the associated form does not vanish),  $I_2(\mathbf{f}) = 0$ if and only if $I_2(f) = 0$,  and $\Delta(\mathbf{f}) = 0$ if and only if $I_3(f) = 0$.

Formula (\ref{jtransfbinquar}) shows that the map $\phi$ extends to the automorphism $\tilde\phi$ of $\PP^1$ given by
$$
\zeta\mapsto\frac{\zeta}{\zeta-1}.
$$
Clearly, one has $\tilde\phi^{\,2}=\hbox{id}$, that is, $\tilde\phi$ is an involution. It preserves $\PP^1\setminus\{1,\infty\}$, which corresponds to the duality between the orbits $O(q_t)$ and $O(q_{-12/t})$ for $t\ne 0,\pm 6$ noted above. Further, $\tilde\phi(1)=\infty$, which agrees with \eqref{Jeq1} and the fact that $O_3$ is mapped onto $O_1$. We also have $\tilde\phi(\infty)=1$, but this identity has no interpretation at the level of orbits.  Indeed, ${\mathbb \Phi}$ cannot be equivariantly extended to an involution $\PP(\cQ_2^4)^{\ss} \to \PP(\cQ_2^4)^{\ss}$ as the fiber of the quotient $\PP(\cQ_2^4)^{\ss}\gitq\SL_2$ over the point at infinity contains $O_1$, which cannot be mapped onto $O_3$ since $\dim O_1<\dim O_3$.

\subsection{Ternary cubics} \label{S:cubics}
Let $n=3$, $d=3$. Every nondegenerate  ternary cubic is linearly equivalent to a cubic of the form
\begin{equation}
c_t(z_1,z_2,z_3):=z_1^3+z_2^3+z_3^3+tz_1z_2z_3,\quad t^3\ne -27\label{ct}
\end{equation} 
(see, e.g.,~Theorem 1.3.2.16 in \cite{Sc}). The associated form of $c_t$ is easily found to be
\begin{equation}
{\mathbf c}_t(z_1,z_2,z_3):=-\frac{1}{24(t^3+27)}(tz_1^3+tz_2^3+tz_3^3-18z_1z_2z_3).\label{bfct}
\end{equation}
For $t\ne 0$, $t^3\ne 216$ the cubic ${\mathbf c}_t$ is nondegenerate, and in this case the associated form of ${\mathbf c}_t$ is proportional to $c_t$, hence $\mathbb{\Phi}^2 (c_t) = c_t$.  Below we will see that the exceptional cubics $c_0$, $c_{6\tau}$, with $\tau^3=1$, are pairwise linearly equivalent.

It is well-known (see, e.g.,~Theorem 1.3.2.16 in \cite{Sc}) that $\PP(\cQ_3^3)^{\ss}$ is the union of $\XX_3^3$ (which coincides with $\PP(\cQ_3^3)^{\s}$) and the following three orbits that consist of strictly semistable forms: ${\rm O}_1:=O(z_1z_2z_3)$, ${\rm O}_2:=O(z_1z_2z_3+z_3^3)$, ${\rm O}_3:=O(z_1^3+z_1^2z_3+z_2^2z_3)$ (the cubics lying in ${\rm O}_3$ are called nodal). The dimensions of the orbits are 6, 7 and 8, respectively. Observe that ${\rm O}_1$ is closed in $\PP(\cQ_3^3)^{\ss}$ and is contained in the closures of each of ${\rm O}_2$, ${\rm O}_3$. We then see that ${\mathbb \Phi}$ maps $\XX_3^3$ onto $\PP(\cQ_3^3)^{\ss}\setminus ({\rm O}_2\cup {\rm O}_3\cup {\rm O}_4)$, where ${\rm O}_4:=O(c_0)$ (as explained below, ${\rm O}_4$ also contains the other exceptional cubics $c_{6\tau}$, with $\tau^3=1$). Further, note that the 8-dimensional orbit ${\rm O}_4$ is mapped by ${\mathbb \Phi}$ onto the 6-dimensional orbit ${\rm O}_1$ (thus the morphism of the stabilizers of $c_0$ and ${\mathbb{\Phi}(c_0)}$ is an inclusion of a finite group into a two-dimensional group). Hence, ${\mathbb \Phi}$ restricts to an equivariant involutive automorphism of $\XX_3^3\setminus {\rm O}_4$, which for $t\ne 0$, $t^3\ne 216$ establishes a duality between the cubics $C\cdot c_t$ and $(C^{-1})^T\cdot c_{-18/t}$ with $C\in\SL_3$, therefore between the orbits $O(c_t)$ and $O(c_{-18/t})$ (see (\ref{equivarianceee})).

To determine the induced map $\phi$ of GIT quotients, we recall that the algebra of invariants $\CC[\cQ_3^3]^{\SL_3}$ is generated by the two Aronhold invariants ${\rm I}_4$, ${\rm I}_6$, where, as before, the subscripts indicate the degrees (see pp.~381--389 in \cite{El}). One has
\begin{equation}
\Delta={\rm I}_6^2+64\, {\rm I}_4^3,\label{discrtercub}
\end{equation}
and for a ternary cubic of the form
\begin{equation}
f(z_1,z_2,z_3)=az_1^3+bz_2^3+cz_3^3+6dz_1z_2z_3\label{generaltercubic}
\end{equation}
the values of ${\rm I}_4$ and ${\rm I}_6$ are calculated as 
\begin{equation}
\begin{array}{l}
{\rm I}_4(f)=abcd-d^4,\quad{\rm  I}_6(f)=a^2b^2c^2-20abcd^3-8d^6.\label{form11}
\end{array}
\end{equation}
It then follows that the algebra $\CC[X_3^3]^{\GL_3}\simeq\CC[\XX_3^3]^{\SL_3}$ is generated by the invariant
\begin{equation}
{\rm J}:=\frac{64\,{\rm I}_4^3}{\Delta}.\label{form21}
\end{equation}
Hence, the quotient $X_3^3\gitq\GL_3\simeq\XX_3^3\gitq\SL_3$ is the affine space $\CC$, and $\PP(\cQ_3^3)^{\ss}\gitq\SL_3$ is identified with $\PP^1$, where ${\rm O}_1$, ${\rm O}_2$, ${\rm O}_3$ project to the point at infinity in $\PP^1$. 

Further, from formulas \eqref{ct}, (\ref{discrtercub}), \eqref{form11}, \eqref{form21} we find 
\begin{equation}
{\rm J}(c_t)=-\frac{t^3(t^3-216)^3}{2^63^3(t^3+27)^3}\quad\hbox{for all $t$ with $t^3\ne -27$.}\label{form31}
\end{equation}
From identity \eqref{form31} one obtains
\begin{equation}
{\rm J}(c_0)={\rm J}(c_{6\tau})=0\quad\hbox{for $\tau^3=1$,}\label{Jeq11}
\end{equation}
which implies that the orbit ${\rm O}_4$ is given by the condition ${\rm J}=0$ and that the four cubics $c_0$, $c_{6\tau}$ are indeed pairwise linearly equivalent.

Using \eqref{bfct}, \eqref{form31} we see
\begin{equation}
{\rm J}({\mathbf c}_t)=\frac{1}{{\rm J}(c_t)}\quad\hbox{for all $t\ne 0$ with $t^3\ne 216$.}\label{jtransftercubics}
\end{equation}
Furthermore, the calculations leading to (\ref{jtransftercubics}) also yield the following identities for any $f \in X_3^3$:
\begin{equation}
\begin{array}{l}
\displaystyle{\rm I}_4(\mathbf{f}) = -\frac{1}{2^{12} 3^{12} \Delta(f)}, \,\,
{\rm I}_6(\mathbf{f}) = -\frac{{\rm I}_6(f)}{2^{15} 3^{18} \Delta(f)^2},\\
\vspace{-0.3cm}\\
\displaystyle\Delta(\mathbf{f}) = -\frac{{\rm I}_4(f)^3}{2^{24} 3^{36} \Delta(f)^4}.
\end{array}\label{formauxcubics}
\end{equation}
Hence, we obtain: ${\rm I}_4(\mathbf{f}) \neq 0$ (that is, the degree 4 Aronhold invariant of the associated form does not vanish), ${\rm I}_6(\mathbf{f}) = 0$ if and only if  ${\rm I}_6(f) = 0 $, and $\Delta(\mathbf{f}) = 0$ if and only if ${\rm I}_4(f) = 0$.  

Formula (\ref{jtransftercubics}) shows that the map $\phi$ extends to the involutive automorphism $\tilde\phi$ of $\PP^1$ given by
$$
\zeta\mapsto\frac{1}{\zeta}.
$$
This involution preserves $\PP^1\setminus\{0,\infty\}$, which agrees with the duality between the orbits $O(c_t)$ and $O(c_{-18/t})$ for $t\ne 0$, $t^3\ne 216$ established above. Next, $\tilde\phi(0)=\infty$, which corresponds to \eqref{Jeq11} and the facts that ${\rm O}_4$ is mapped onto ${\rm O}_1$, and that ${\rm I}_4(f) = 0$ implies $\Delta(\mathbf{f}) =0$. Also, one has $\tilde\phi(\infty)=0$, but this identity cannot be illustrated by a correspondence between orbits.  Indeed, ${\mathbb \Phi}$ cannot be equivariantly extended to an involution $\PP(\cQ_3^3)^{\ss} \to \PP(\cQ_3^3)^{\ss}$ as  the fiber of the quotient $\PP(\cQ_3^3)^{\ss}\gitq\SL_2$ over the point at infinity contains $O_1$, which cannot be mapped onto $O_4$ since $\dim O_1<\dim O_4$.

\begin{remark}\label{invsys}
We note that a cubic proportional to \eqref{bfct} previously appeared in \cite{Em} (see p.~405 therein) as a Macaulay inverse system for the Milnor algebra $M_{c_t}$, but it has never been studied systematically. In fact, we now know (see Corollary 3.3 in \cite{AI1}) that the associated form of any $f\in X_n^d$ is an inverse system for $M_f$ when regarded as an element of $\cQ_n^{n(d-2)}$. This result has been instrumental in our recent work on the morphism $\Phi$ including the progress on  Conjecture \ref{conj2}, and it will be also utilized in the proof of Theorem \ref{main} below (see Lemma \ref{PhiPsi}). 
For details on inverse systems we refer the reader to \cite{Ma}, \cite{Em}, \cite{I} (the brief survey given in \cite{ER} is also helpful). We also note that, although the Hessian $H(f)$ is utilized in the definition of the associated form ${\mathbf f}$ of $f$, it is in fact very different from ${\mathbf f}$. Indeed, as shown in \cite{DBP}, for binary quartics and ternary cubics, $H(f)$ does not coincide with ${\mathbf f}$ up to projective equivalence except in a few cases (see Propositions 4.1 an 5.1 therein). We will elaborate on this difference in Subsection \ref{S:contra1}.
\end{remark}

If we regard $\XX_3^3$ as the space of elliptic curves, the invariant ${\rm J}$ of ternary cubics translates into the $j$-invariant, and one obtains an equivariant involution on the locus of elliptic curves with nonvanishing $j$-invariant. It is well-known that every elliptic curve can be realized as a double cover of $\PP^1$ branched over four points (see, e.g.,~Exercise 22.37 and Proposition 22.38 in \cite{Ha}). Therefore, it is not surprising that the cases of binary quartics and ternary cubics considered above have many similarities.  What is perhaps surprising though is that the map $\mathbb{\Phi}$ for binary quartics and ternary cubics yields {\it different}\, involutions on $\PP^1$. It is natural to ask whether there exist any other involutions of $\PP^1$ that arise from rational equivariant involutions on $\PP(\cQ_2^4)$ and $\PP(\cQ_3^3)$ as above. The result of the next subsection provides a complete answer to this question.

\subsection{Uniqueness of rational equivariant involutions}\label{uniqueness}
In this subsection we classify rational $\SL_n$-invariant involutions 
$$
\iota\co\PP(\cQ_n^d)\dashrightarrow \PP(\cQ_n^d)
$$
for $n=2$, $d=4$ and $n=3$, $d=3$. Here the equivariance is understood either as
\begin{equation}
\iota(C\cdot f)=C\cdot\iota(f),\label{equivartype1}
\end{equation}
or as
\begin{equation}
\iota(C\cdot f)=(C^{-1})^T\cdot\iota(f),\label{equivartype2}
\end{equation}
where $C\in\SL_n$ and $f$ lies in the domain of $\iota$. The identity morphism and the map ${\mathbb\Phi}$ are respective examples of such involutions. The result below asserts that there are no other possibilities.

\begin{theorem} \label{invclass} For each pair $n=2$, $d=4$ and $n=3$, $d=3$ the following holds:
\begin{enumerate}
\item[{\rm (i)}] the identity morphism is the unique rational involution of\, $\PP(\cQ_n^d)$ satisfying {\rm (\ref{equivartype1})};
\item[{\rm (ii)}]  the morphism ${\mathbb \Phi}$ is the unique rational involution of\, $\PP(\cQ_n^d)$ satisfying {\rm (\ref{equivartype2})}.
\end{enumerate}
\end{theorem}

\begin{proof} Let $n=2$, $d=4$. Recall from Subsection \ref{S:binary-quartics} that every generic binary quartic is linearly equivalent to some quartic $q_t = z_1^4 + t z_1^2z_2^2 + z_2^4$, with $t \in \CC$, $t\ne \pm 2$, and that ${\mathbb \Phi}(q_t) = q_{-12/t}$ if $t\ne 0$. Therefore, in order to establish the theorem, it suffices to prove that in (i) (resp. (ii)) one has $\iota(q_t)=q_t$ (resp. $\iota(q_t)=q_{-12/t}$) for a generic $t$. 

We first obtain part (i).  For a generic $t$ one can write
$$
\iota(q_t) = \alpha_{4,0} z_1^4 + 4 \alpha_{3,1} z_1^3 z_2 + 6 \alpha_{2,2} z_1^2 z_2^2 + 4 \alpha_{1,3} z_1z_2^3 + \alpha_{0,4} z_2^4,
$$
where $\alpha_{i,j} \in \CC[t]$.   Consider $C = \begin{pmatrix} 0 & i \\ i & 0 \end{pmatrix}$. Since $C \cdot q_t = q_t$ for all $t$ and $\iota$ is equivariant, it follows that $\alpha_{4,0} = \alpha_{0,4}$.  Similarly, by considering $C = \left(\hspace{-0.2cm}\begin{array}{rr} -i & 0 \\ 0 & i \end{array}\hspace{-0.1cm}\right)$, we see $\alpha_{3,1} = \alpha_{1,3} = 0$.  Observe now that $\alpha_{4,0}$ does not vanish identically since otherwise $\iota$ would be a constant map. Therefore, one can write
$$
\iota(q_t) = z_1^4 + \alpha z_1^2 z_2^2 + z_2^4,
$$
with $\alpha:= 6\alpha_{2,2} / \alpha_{4,0} \in \CC(t)$.

As $\iota$ is a birational morphism, the assignment $t \mapsto \alpha(t)$ extends to an automorphism of $\PP^1$, i.e.~we have 
\begin{equation}
\alpha(t) = \frac{at+b}{ct+d}\label{isomp1}
\end{equation}
for some $\begin{pmatrix} a & b \\ c & d \end{pmatrix}\in\GL_2$.  Consider $C = \begin{pmatrix} \sqrt{i} & 0 \\ 0 & -i \sqrt{i} \end{pmatrix}$.  By observing that $C\cdot q_t=q_{-t}$ for all $t$ and using the equivariance of $\iota$, we obtain $\alpha(-t) =-\alpha(t)$. This in turn implies that either $b=c=0$ or $a=d=0$. If $b=c=0$, the fact that $\iota$ is an involution leads to $a=\pm d$. Suppose that $a=-d$ hence $\iota(q_t)=q_{-t}$. For
\begin{equation}
C = \displaystyle\frac{1}{\sqrt{2}}\left(\begin{array}{rr} 1 & 1 \\ -1 & 1 \end{array}\right)\label{specialmat1}
\end{equation}
one computes
\begin{equation}
C\cdot q_t=q_{{}_{\frac{-2t+12}{t+2}}}\label{specrel1}.
\end{equation}
The equivariance of $\iota$ then leads to a contradiction. 
If $a=d=0$, then utilizing matrix (\ref{specialmat1}) with relation (\ref{specrel1}) once again, we conclude $b/c =-12$.

Thus, we have obtained:
\begin{equation}
\hbox{either $\iota(q_t) = q_t$ or $\iota(q_t) =q_{-12/t}$ for a generic $t\in\CC$.}\label{tworel}
\end{equation}
Since the second relation in (\ref{tworel}) contradicts equivariance property (\ref{equivartype1}), it follows that $\iota$ is the identity morphism, and part (i) is established.

Notice that an argument analogous to that for part (i) yields relations (\ref{tworel}) for part (ii) as well. Since the first relation in (\ref{tworel}) contradicts equivariance property (\ref{equivartype2}), we obtain $\iota={\mathbb \Phi}$. This concludes the proof for $n=2$, $d=4$.

Suppose now that $n=3$, $d=3$. As stated in Subsection \ref{S:cubics}, every generic ternary cubic is linearly equivalent to some cubic\linebreak $c_t = z_1^3+z_2^3+z_3^3+tz_1z_2z_3$, with $t \in \CC$, $t^3\ne -27$, and one has $\mathbb{\Phi}(c_t) = c_{-18/t}$ if $t\ne 0$.  Thus, to prove the theorem, it suffices to show that in (i) (resp. (ii)) one has $\iota(c_t)=c_t$ (resp. $\iota(c_t)=c_{-18/t}$) for a generic $t$. 

We first obtain part (i).  For a generic $t$ one can write
$$
\begin{array}{lll}
\iota(c_t) &=& \alpha_{3,0,0} z_1^3+ 3 \alpha_{2,1,0} z_1^2z_2+3 \alpha_{1,2,0} z_1z_2^2+\alpha_{0,3,0}z_2^3+\\
\vspace{-0.3cm}\\
&& \hspace{.6cm} 3\alpha_{2,0,1}z_1^2z_3+6\alpha_{1,1,1}z_1z_2 z_3+3\alpha_{0,2,1}z_2^2z_3+\\
\vspace{-0.3cm}\\
&&\hspace{1.7cm}  3\alpha_{1,0,2}z_1z_3^2+3\alpha_{0,1,2}z_2 z_3^2 + \\
\vspace{-0.3cm}\\
&&\hspace{3.2cm}   \alpha_{0,0,3} z_3^3,
\end{array}
$$
where $\alpha_{i,j,k} \in \CC[t]$. Consider
$$
C = \left(\begin{array}{rrr} 0 & -1 & 0 \\ -1 & 0 & 0 \\ 0 & 0 & -1 \end{array}\right).
$$
Since $C \cdot c_t = c_t$ for all $t$ and $\iota$ is equivariant, we immediately see $\alpha_{3,0,0} = \alpha_{0,3,0}$. A similar choice of $C$ yields $\alpha_{0,3,0}=\alpha_{0,0,3}$, thus we have $\alpha_{3,0,0} = \alpha_{0,3,0}=\alpha_{0,0,3}$.

Next, let $\tau\ne 1$ satisfy $\tau^3=1$ and consider
$$
C = \left(\begin{array}{rrr} \tau & 0 & 0 \\ 0 & \tau^2 & 0 \\ 0 & 0 & 1 \end{array}\right).
$$
Then again $C \cdot c_t = c_t$ for all $t$, and the equivariance of $\iota$ implies
$$
\alpha_{2,1,0}=\alpha_{1,2,0}=\alpha_{2,0,1}=\alpha_{0,2,1}=\alpha_{1,0,2}=\alpha_{0,1,2}=0.
$$
It now follows that $\alpha_{3,0,0}$ does not vanish identically, thus one can write
\begin{equation}
\iota(c_t) = z_1^3+z_2^3+z_3^3 + \alpha z_1z_2z_3,\label{iotaofct}
\end{equation}
with $\alpha:= 6\alpha_{1,1,1} / \alpha_{3,0,0} \in \CC(t)$ having the form (\ref{isomp1}). 

Further, setting
\begin{equation}
C = \left(\begin{array}{rrr} \tau^{1/3} & 0 & 0 \\ 0 & \tau^{1/3} & 0 \\ 0 & 0 & \displaystyle\frac{1}{\tau^{2/3}} \end{array}\right)\label{specmatr4}
\end{equation}
and observing that $C\cdot c_t=c_{\tau t}$ for all $t$, we obtain using the equivariance of $\iota$ that $\alpha(\tau t) =\tau\alpha(t)$. This implies $b=c=0$, hence $a=\pm d$, with the case $a=d$ yielding that $\iota$ is the identity. 
Assume that $a=-d$, thus $\iota(c_t)=c_{-t}$. If
\begin{equation}
C = \displaystyle\frac{1}{(3(\tau^2-\tau))^{1/3}}\left(\begin{array}{ccc} 1 & 1 & 1 \\ \tau & \tau^2 & 1\\ \tau^2 &\tau & 1 \end{array}\right),\label{specialmat2}
\end{equation}
then
\begin{equation}
C\cdot c_t=c_{{}_{\frac{-3t+18}{t+3}}},\label{specrel2}
\end{equation}
and the equivariance of $\iota$ leads to a contradiction. This concludes the proof for part (i).

We will now obtain part (ii). The same method leads to formula (\ref{iotaofct}) with $\alpha$ as in (\ref{isomp1}). Using matrix (\ref{specmatr4}) and the equivariance of $\iota$, we see that in this case $\alpha(\tau t)=\tau^2\alpha(t)$, which yields $a=d=0$. Utilizing matrix (\ref{specialmat2}) and appealing to relation (\ref{specrel2}), we conclude $b/c =-18$. Therefore, $\iota(q_t) = q_{-18/t}$ for a generic $t\in\CC$. Hence $\iota={\mathbb \Phi}$, which completes the proof of the theorem.\end{proof}

\subsection{Projective duality}
In this subsection we will see that for $n=2$, $d=4$ and $n=3$, $d=3$ the map ${\mathbb \Phi}$, and therefore the orbit duality induced by ${\mathbb \Phi}$, can be understood via projective duality. We will now briefly recall this classical construction. For details the reader is referred to the comprehensive survey \cite{T}. 

Let $W$ be a complex vector space and $\PP(W)$ its projectivization. The dual projective space $\PP(W)^*$ is the algebraic variety of all hyperplanes in $W$, which is canonically isomorphic to $\PP(W^*)$. Let $X$ be an irreducible subvariety of $\PP(W)$ and $X_{\reg}$ the set of its regular points. Consider the affine cone $\hat X\subset W$ over $X$. For every $x\in X_{\reg}$ choose a point $\hat x\in\hat X$ lying over $x$. The cone $\hat X$ is regular at $\hat x$, and we consider the tangent space $T_{\hat x}(\hat X)$ to $\hat X$ at $\hat x$. Identifying $T_{\hat x}(\hat X)$ with a subspace of $W$, we now let $H_x$ be the collection of all hyperplanes in $W$ that contain $T_{\hat x}(\hat X)$ (clearly, this collection is independent of the choice of $\hat x$ over $x$). Regarding each hyperplane in $H_x$ as a point in $\PP(W)^*$, we obtain the subset
$$
H:=\bigcup_{x\in X_{\reg}}H_x \subset \PP(W)^*.
$$
The Zariski closure $X^*$ of $H$ in $\PP(W)^*$ is then called the variety dual to $X$. Canonically identifying $\PP(W)^{**}$ with $\PP(W)$, one has the reflexivity property $X^{**}=X$. Furthermore, if $X$ is a hypersurface, there exists a natural map from $X_{\reg}$ to $X^*$, as follows:
$$
\varphi: X_{\reg}\to X^*,\quad x\mapsto T_{\hat x}(\hat X)\subset W,
$$
where $\hat x\in\hat X$ is related to $x\in X_{\reg}$ as above.  

Observe now that in each of the two cases $n=2$, $d=4$ and $n=3$, $d=3$, for $f\in \XX_n^d$ the orbit $O(f)$ is a smooth irreducible hypersurface in $\XX_n^d$, thus its closure $\overline{O(f)}$ in $\PP(\cQ_n^d)$ is an irreducible (possibly singular) hypersurface. Therefore, one can consider the map
\begin{equation}
\varphi_f: \overline{O(f)}_{\reg}\to \PP(\cQ_n^d)^*\label{mapvarphismall}
\end{equation}
constructed as above. Also, recall that for $n=2$, $d=4$ or $n=3$, $d=3$, the morphism $\Phi$ descends to the morphism
$$
{\mathbb \Phi}\co \XX_n^d \to \PP(\cQ_n^{d*}).
$$
We are now ready to state the first main result of the paper, which relates these two maps.

\begin{theorem}\label{main}  Let $n=2$, $d=4$ or $n=3$, $d=3$.  For every $f\in\XX_n^d$ the restrictions ${\mathbb \Phi}\big|_{O(f)}$ and $\varphi_f\big|_{O(f)}$ coincide upon the canonical identification $\PP(\cQ_n^{d*})=\PP(\cQ_n^d)^*$.
\end{theorem}

This theorem provides a clear explanation of the duality for orbits of binary quartics and ternary cubics that we observed earlier in this section. Indeed, let first $n=2$, $d=4$. Then the theorem yields that for $t\ne 0,\pm 6$ one has $\overline{O(q_t)}^{\,*}\simeq\overline{O(q_{-12/t})}$ and $\overline{O}_3^{\,*}\simeq\overline{O}_1$. By reflexivity it then follows that $\overline{O}_1^{\,*}\simeq\overline{O}_3$. However, since $O_1$ is not a hypersurface, there is no natural map from $\overline{O}_1$ to its dual. This fact corresponds to the impossibility to extend ${\mathbb \Phi}$ equivariantly to $O_1$.

Analogously, for $n=3$, $d=3$, the theorem implies that for $t\ne 0$ and $t^3\ne 216$ we have $\overline{O(c_t)}^{\,*}\simeq\overline{O(c_{-18/t})}$ and $\overline{{\rm O}}_4^{\,*}\simeq\overline{{\rm O}}_1$. By reflexivity one then has $\overline{{\rm O}}_1^{\,*}\simeq\overline{{\rm O}}_4$. Again, since ${\rm O}_1$ is not a hypersurface, there is no natural map from $\overline{{\rm O}}_1$ to its dual. This agrees with the nonexistence of an equivariant extension of ${\mathbb \Phi}$ to ${\rm O}_1$.

\subsection{Proof of Theorem \ref{main}}\label{proofmain}

First, let $n\ge 2$ and $d\ge 3$ be arbitrary. For a complex vector space $W$, let $\Gr(k,W)$ denote the Grassmannian of $k$-dimensional subspaces of $W$. Notice that $\Gr(\dim_{\CC}W-1,W)$ coincides with $\PP(W)^*$. It follows, for instance, from Corollary 3.3 in \cite{St} (see also the proof of Lemma \ref{twomaps} below), that for any $f\in\XX_n^d$ the dimension of the subspace of $\cQ_n^d$ spanned by the forms $z_if_j$, with $i,j=1,\dots,n$, is equal to $n^2$. 

We then define two maps from $\XX_n^d$ to $\Gr(n^2,\cQ_n^d)$ as
\begin{equation}
\psi_1:f\mapsto \cQ_n^1 f_1+\dots+\cQ_n^1f_n\subset\cQ_n^d\label{mappsi}
\end{equation}
and
\begin{equation}
\psi_2:f\mapsto T_f(\GL_n\cdot f),\label{mapphismall}
\end{equation}
where in the right-hand sides the element $f$ of $\XX_n^d$ is regarded as a form in $X_n^d$ and $T_f(\GL_n\cdot f)$ as a subspace of $\cQ_n^d$. We will now show that these two maps are in fact equal.

\begin{lemma}\label{twomaps} \it One has $\psi_1=\psi_2$.
\end{lemma}

\begin{proof}
For $f\in X_n^d$ let $\sigma_f:\GL_n\to\cQ_n^d$ be the morphism defined by $\sigma_f(C)=C\cdot f$. Then the tangent space $T_f(\GL_n\cdot f)$ is the image of the differential $d\sigma_f(e)$ of $\sigma_f$ at the identity element $e\in\GL_n$. 

Let $\{{\tt E}_{ij}\}$ be the standard basis in the Lie algebra ${\mathfrak{gl}}_n$ of $\GL_n$, where ${\tt E}_{ij}$ is the matrix whose $(i,j)$th element is 1 and all other elements are zero. Then, if we regard $d\sigma_f(e)$ as a linear transformation from ${\mathfrak{gl}}_n$ to $\cQ_n^d$, it is easy to compute that it maps ${\tt E}_{ij}$ to $-z_jf_i$. This shows that $\psi_1$ and $\psi_2$ indeed coincide as required. \end{proof} 

Next, consider the map
\begin{equation} \label{eqn-psi}
{\mathbb \Psi} \co \XX_n^d\to\PP(\cQ_n^{n(d-2)})^*,\quad f\mapsto W_f,
\end{equation}
where $W_f$ is the hyperplane in $\cQ_n^{n(d-2)}$ defined in \eqref{subspace}. We will now relate the morphism ${\mathbb \Phi}$ to this map. 

\begin{lemma} \label{PhiPsi} \it The morphisms
$$
{\mathbb \Phi}  \co \XX_n^d\to\PP(\cQ_n^{n(d-2)*})\,\,\hbox{and}\,\,\, {\mathbb \Psi} \co \XX_n^d\to\PP(\cQ_n^{n(d-2)})^*
$$
coincide upon the canonical identification $\PP(\cQ_n^{n(d-2)*})=\PP(\cQ_n^{n(d-2)})^*$.
\end{lemma}

\begin{proof}
By formula \eqref{assocformexpp}, for $f\in\XX_n^d$ we see that ${\mathbb \Phi}(f)$ is the hyperplane in $\cQ_n^{n(d-2)}$ that consists of all forms
\begin{equation}
g=\sum_{i_1+\dots+i_n=n(d-2)}\frac{(n(d-2))!}{i_1!\dots i_n!}a_{i_1,\dots,i_n}z_1^{i_1}\dots z_n^{i_n}\label{formg}
\end{equation}
satisfying the condition
\begin{equation}
\sum_{i_1+\dots+i_n=n(d-2)}\frac{1}{i_1!\dots i_n!}\mu_{i_1,\dots,i_n}(f)a_{i_1,\dots,i_n}=0,\label{condPhi}
\end{equation}
where $f$ is regarded as an element of $X_n^d$ and $\mu_{i_1,\dots,i_n}(f)$ are the coefficients from \eqref{assocformexpp}.

On the other hand, by Corollary 3.3 of \cite{AI1}, considering $f$ as a form in $X_n^d$ and $\Phi(f)$ as an element of $\cQ_n^{n(d-2)}$ (see Remark \ref{olddef}), we have that $\Phi(f)$ is a Macaulay inverse system for the Milnor algebra $M_f$. This means that the ideal
$(f_1,\dots,f_n)$ in $\CC[z_1,\dots,z_n]$ coincides with $\Ann(\Phi(f))$, where for any polynomial $h\in\CC[z_1,\dots,z_n]$ the annihilator $\Ann(h)$ of $h$ is defined as
$$
\Ann(h):=\left\{p\in \CC[z_1,\dots,z_n] \mid p\left(\frac{\partial}{\partial z_1},\dots,\frac{\partial}{\partial z_n}\right)(h)=0\right\}.
$$
Therefore, $W_f=\Ann(\Phi(f))\cap\cQ_n^{n(d-2)}$, which immediately implies that $W_f$ consists of all forms $g\in\cQ_n^{n(d-2)}$ as in \eqref{formg} satisfying \eqref{condPhi}. This shows ${\mathbb \Phi}={\mathbb \Psi}$.
\end{proof}

\begin{proof}[Proof of Theorem {\rm \ref{main}}] In the cases $n=2$, $d=4$ and $n=3$, $d=3$ we have $\Gr(n^2,\cQ_n^d)=\PP(\cQ_n^d)^*$, and   Lemma \ref{twomaps} shows that the two morphisms $\psi_1, \psi_2: \XX_n^d \to \PP(\cQ_n^d)^*$ defined in \eqref{mappsi} and \eqref{mapphismall} are equal. Further, $\psi_1 = {\mathbb \Psi}$ by \eqref{mappsi} and \eqref{eqn-psi}, hence Lemma \ref{PhiPsi} implies $\psi_1 = {\mathbb \Phi}$, which yields $\psi_2={\mathbb \Phi}$.  Moreover, for every $f\in\XX_n^d$ the map $\varphi_f\big|_{O(f)}$ from \eqref{mapvarphismall} is identical to $\psi_2\big|_{O(f)}$. It then follows that ${\mathbb \Phi}\big|_{O(f)}$ and $\varphi_f\big|_{O(f)}$ coincide for all $f\in\XX_n^d$, which establishes the theorem.
\end{proof}

\section{The contravariants defined by $\Phi$}\label{S:contravariant}
\setcounter{equation}{0}

In this section, we give an algebraic description of the map ${\mathbb \Phi}$ for binary quartics and ternary cubics, which utilizes classical covariants and contravariants. Such descriptions can be produced in other situations as well, and, in order to further illustrate our method, we also discuss the case of binary quintics. 

\subsection{Covariants and contravariants}
Recall that a polynomial\linebreak $\Gamma \in\CC[\cQ_n^d\times\CC^n]$ is said to be a covariant of forms in $\cQ_n^d$ if for all $f\in\cQ_n^d$, $z\in\CC^n$ and $C\in\GL_n$ the following holds:
$$
\Gamma(f,z)=(\det C)^k\, \Gamma(C\cdot f,C\cdot z),
$$
where $k$ is an integer called the weight of\, $\Gamma$. Every homogeneous component of\, $\Gamma$ with respect to $z$ is automatically homogeneous with respect to $f$ and is also a covariant. Such covariants are called homogeneous and their degrees with respect to $f$ and $z$ are called the degree and order, respectively.  

We may consider a homogenous covariant $\Gamma$ of degree $K$ and order $D$ as the $\SL_n$-equivariant morphism
$$
\begin{aligned}
\cQ_n^d & \to \cQ_n^D,\\
 f & \mapsto (z \mapsto \Gamma(f,z))
 \end{aligned}$$
 of degree  $K$ with respect to $f$, which maps a form $f\in\cQ_n^d$ to the form in $\cQ_n^D$ whose evaluation at $z$ is $\Gamma(f,z)$.  We will abuse notation by using the same symbol to denote both an element in $\CC[\cQ_n^d \times \CC^n]$ and the corresponding morphism $ \cQ_n^d \to \cQ_n^D$. Also, we write $\Gamma(f)$ for the form  $z \mapsto \Gamma(f,z)$.
 
Covariants independent of $z$ (i.e.~of order $0$) are called relative invariants. For example, the pairs of functions $I_2$, $I_3$ and ${\rm I}_4$, ${\rm I}_6$ introduced in Section \ref{result} are relative invariants of binary quartics and ternary cubics, respectively. Also, note that the discriminant $\Delta$ is a relative invariant of forms in $\cQ_n^d$ of weight $d(d-1)^{n-1}$ (see Chapter 13 in \cite{GKZ}).

Analogously, a polynomial $\Lambda \in\CC[\cQ_n^d\times \CC^{n*}]$ is said to be a contravariant of forms in $\cQ_n^d$  if for all $f\in\cQ_n^d$, $z^*=(z_1^*,\dots,z_n^*) \in \CC^{n*}$ and $C\in\GL_n$ one has
$$
\Lambda(f,z^*)=(\det C)^k\, \Lambda(C\cdot f,C \cdot z^*),
$$
where $k$ is a (nonnegative) integer called the weight of\, $\Lambda$ and
$$
C \cdot z^*:= (z_1^*,\dots,z_n^*)\,C^{-1}.
$$
Again, every contravariant splits into a sum of homogeneous ones, and for a homogeneous contravariant its degrees with respect to $f$ and $z^*$ are called the degree and class, respectively. 

We may consider a homogenous contravariant $\Lambda$ of degree $K$ and class $D$ as the $\SL_n$-equivariant morphism
$$
\begin{aligned}
\cQ_n^d & \to \Sym^D(\CC^n),\\
 f & \mapsto (z^* \mapsto \Lambda(f,z^*))
\end{aligned}
$$
of degree $K$ with respect to $f$. Upon the standard identification\linebreak $\Sym^D(\CC^n)=(\Sym^D \CC^{n*})^*=\cQ_n^{D*}$ induced by the polar pairing, this morphism can be regarded as a map from $\cQ_n^d$ to $\cQ_n^{D*}$. As above, we will abuse notation by using the same symbol to denote both an element in $\CC[\cQ_n^d \times \CC^{n*}]$ and the corresponding morphisms $\cQ_n^d \to \Sym^D(\CC^n)$, $\cQ_n^d \to \cQ_n^{D*}$. Also, we write $\Lambda(f)$ for both the element of $\Sym^D(\CC^n)$ and the element of $\cQ_n^{D*}$ arising from $f$ and $\Lambda$.

If $n=2$, every homogeneous contravariant $\Lambda$ yields a homogenous covariant $\hat{\Lambda}$ via the formula
\begin{equation}
\hat{\Lambda}(f)(z_1,z_2) := \Lambda(f)(-z_2, z_1),\quad (z_1,z_2)\in\CC^2,\label{relcovcontrav}
\end{equation}
where $(-z_2, z_1)$ is viewed as a point in $\CC^{2*}$. Analogously, every homogeneous covariant $\Gamma$ gives rise to a homogenous contravariant $\tilde{\Gamma}$ via the formula
\begin{equation}
\tilde{\Gamma}(f)(z_1^*,z_2^*) := \Gamma(f)(z_2^*, -z_1^*),\quad (z_2^*,-z_1^*)\in\CC^{2*},\label{relcovcontrav1}
\end{equation}
where $(z_2^*,-z_1^*)$ is regarded as a point in $\CC^{2}$. Under these correspondences the degree and order of a covariant translate into the degree and class of the corresponding contravariant and vice versa.

\subsection{The contravariant defined by $\Phi$}
Recall that the morphism $\Phi$ is a map
$$
\Phi \co X_n^d \to \cQ_n^{n(d-2)*}
$$
defined on the locus $X_n^d$ of nondegenerate forms.  The coefficients $\mu_{i_1, \ldots, i_n}$ that determine $\Phi$ (see \eqref{assocformexpppp}, (\ref{assocformexpp})) are elements of the coordinate ring $\CC[X_n^d] = \CC[\cQ_n^d]_{\Delta}$.  Let $p_{i_1,\dots,i_n}$ be the minimal integer such that $\Delta^{p_{i_1,\dots,i_n}}\cdot\mu_{i_1,\dots,i_n}$ is a regular function on ${\mathcal Q}_n^d$ (see formula \eqref{formulaformus}) and
$$
p=\max\{p_{i_1,\dots,i_n}\mid i_1+\dots+i_n=n(d-2)\}.
$$
Then the product $\Delta^p \Phi$ defines the following morphism
$$
\Delta^p \Phi \co \cQ_n^d \to \cQ_n^{n(d-2)*}, \quad f \mapsto \Delta(f)^p \Phi(f),
$$
which is a contravariant of weight $pd(d-1)^{n-1}-2$ by Proposition \ref{equivariance}.  Since the class of $\Delta^p \Phi$ is $n(d-2)$, it follows that its degree is equal to $np(d-1)^{n-1}-n$. Notice that this last formula implies $p>0$ as the degree of a contravariant is always nonnegative.

In this subsection, we show that for binary and ternary forms one has $p=1$. It then follows that for $n=2,3$ the product $\Delta \Phi \co \cQ_n^d \to\cQ_n^{n(d-2)*}$ is a contravariant of degree $n(d-1)^{n-1}-n$.  In Subsections \ref{S:contra1}--\ref{S:contra3}, we study this contravariant explicitly for binary quartics, binary quintics and ternary cubics in terms of well-known classical contravariants.

\begin{proposition}\label{moreprecise} 
If $n=2,3$ and $d \ge 3$, then $p=1$.
\end{proposition}

\begin{proof} First, let $n$ be arbitrary. Recall that for $f\in X_n^d$ the subspace $W_f$ introduced in \eqref{subspace} has codimension 1, and the line spanned by the Hessian $H(f)$ is complementary to it in the vector space $\cQ_n^{n(d-2)}$. Let $K:=\dim_{\CC}{\mathcal Q}_n^{n(d-2)-d+1}$ and ${\tt m}_1,\dots,{\tt m}_K$ be the standard monomial basis in ${\mathcal Q}_n^{n(d-2)-d+1}$. Then $W_f$ is spanned by the products $f_i\,{\tt m}_j$,\linebreak $i=1,\dots,n$, $j=1,\dots,K$. Choose a basis ${\mathbf e}_k(f):= f_{i_{{}_k}}\,{\tt m}_{j_{{}_k}}$ in $W_f$, with $k=1,\dots,N-1$, where $N:=\dim_{\CC}\cQ_n^{n(d-2)}$. We note that the indices $i_k$, $j_k$ can be assumed to be independent of the form $f$ if it varies in some Zariski open subset $U$ of $X_n^d$, and from now on we assume that this is the case. Then for every $g\in\cQ_n^{n(d-2)}$ there are $\alpha_k(f,g)\in\CC$, with $i=1,\dots,N-1$, and $\gamma(f,g)\in\CC$ such that
\begin{equation}
\alpha_1(f,g){\mathbf e}_1(f)+\dots+\alpha_{N-1}(f,g){\mathbf e}_{N-1}(f)+\gamma(f,g) H(f)=g.\label{keyeq}
\end{equation}
Notice that $\gamma(f,z_1^{i_1}\dots z_n^{i_n})=\mu_{i_1,\dots,i_n}(f)$ (see (\ref{assocformexpppp})).

We now expand both sides of \eqref{keyeq} with respect to the standard monomial basis of $\cQ_2^{2(d-2)}$. As a result, we obtain a linear system of $N$ equations with the $N$ unknowns $\alpha_k(f,g)$, $\gamma(f,g)$, $k=1,\dots,N-1$. Let $A(f)$ be the matrix of this system and $D(f):=\det A(f)$. Clearly, the entries of the first $N-1$ columns of the matrix are linear functions of the coefficients of the form $f$, whereas the entries of the $N$th column are homogeneous polynomials of degree $n$ of these coefficients. Therefore, $D$ is a homogeneous polynomial of degree $\delta_1:=N+n-1$ on $\cQ_n^d$. Furthermore, since for every $f$ and $g$ the system has a solution, $D$ does not vanish on $U$, and the solution can be found by applying Cramer's rule. It then follows that the degree of the minimal denominator of $\mu_{i_1,\dots,i_n}$ does not exceed $\delta_1$.

At the same time, $\Delta^{p_{i_1,\dots,i_n}}\cdot\mu_{i_1,\dots,i_n}$ is a regular function on ${\mathcal Q}_n^d$ (recall that $p_{i_1,\dots,i_n}$ is the minimal integer with this property). It is well-known that $\Delta$ is an irreducible homogeneous polynomial of degree\linebreak $\delta_2:=n(d-1)^{n-1}$ on ${\mathcal Q}_n^d$ (the irreducibility of $\Delta$ can be observed by considering an incidence variety as on p.~169 in \cite{Mu}). Therefore, the degree of the minimal denominator of $\mu_{i_1,\dots,i_n}$ is $p_{i_1,\dots,i_n}\cdot\delta_2$. 

We will now show that for $n=2$ and $n=3$ one has
\begin{equation}
\delta_1<2\delta_2.\label{12ineq}
\end{equation}
Indeed, using \eqref{dimform}, we obtain
$$
\delta_1=\frac{(n(d-1)-1)!}{(n-1)!(n(d-2))!}+n-1,
$$
which yields
$$
\delta_1=\left\{\begin{array}{ll} 2d-2, & \hbox{if $n=2$,}\\
\vspace{-0.1cm}\\
\displaystyle\frac{9d^2-27d+24}{2}, & \hbox{if $n=3$.}
\end{array}\right.
$$
Then for $n=2$ we see $\delta_1=\delta_2$, and after some calculations it follows that for $n=3$ inequality \eqref{12ineq} holds. This implies $p_{i_1,\dots,i_n}\le 1$, and the proof is complete.\end{proof}

\begin{remark}\label{numberofvar}\rm For $n=2$ in the above proof one has $2K=N-1$, hence $D(f)$ does not vanish for all $f\in X_2^d$. In other words, 
$$
\{f\in{\mathcal Q}_2^d: D(f)= 0\}\subset \{f\in{\mathcal Q}_2^d: \Delta(f)= 0\}.\label{zerosets}
$$
Furthermore, in this case $\delta_1=\delta_2=2d-2$. It then follows that $D$ and $\Delta$ coincide up to a scalar factor. Thus, an interesting byproduct of the proof of Proposition \ref{moreprecise} is the fact that for any binary form $f$ the discriminant of $f$ can be computed as the determinant of $A(f)$ up to scale.
\end{remark}

\begin{remark}\label{quarternary}\rm It is not hard to see that inequality \eqref{12ineq} also holds for $n=4$ and $d\le 6$, with $n=4$, $d=7$ being the first case when it fails. In fact, arguing as in the proof of Proposition \ref{moreprecise}, one can derive a certain estimate in terms of $n$ and $d$ on the power of $\Delta$ that can occur in the minimal denominator of $\mu_{i_1,\dots,i_n}$. 
\end{remark}

\subsection{Binary quartics}\label{S:contra1}
Let $n=2$, $d=4$. In this case $\Delta \Phi$ is a contravariant of weight 10, degree 4 and class 4.  We have the following identity of covariants of weight 6 (see (\ref{relcovcontrav})):
\begin{equation}
\widehat {\Delta \Phi}= \frac{1}{2^7 3^3}I_2 H -  \frac{1}{2^4}I_3{\mathbf{id}},\label{covar1}
\end{equation}
where $H$ is the Hessian, $I_2, I_3$ the invariants of degrees $2,3$, respectively, defined in Subsection \ref{S:binary-quartics}, and ${\mathbf{id}}:f\mapsto f$ the identity covariant. To verify \eqref{covar1}, it is sufficient to check it for the quartics $q_t$ introduced in \eqref{qt}. For these quartics the validity of \eqref{covar1} is a  consequence of formulas \eqref{bfqt}--\eqref{form1}. 

Observe that formula \eqref{covar1} is not a result of mere guesswork; it follows naturally from Proposition \ref{moreprecise} and an explicit description of the algebra of covariants of binary quartics. Indeed, this algebra is generated by $I_2$, $I_3$, the Hessian $H$ (which has degree 2 and order 4), the identity covariant ${\mathbf{id}}$ (which has degree 1 and order 4), and one more covariant of degree 3 and order 6 (see pp.~180--181 in \cite{El}). Therefore $\widehat {\Delta \Phi}$, being a covariant of degree 4 and order 4 by Proposition \ref{moreprecise}, is necessarily a linear combination of $I_2H$ and $I_3{\mathbf {id}}$. The coefficients in the linear combination can be determined by computing $\Delta \Phi$, $I_2H$ and $I_3{\mathbf {id}}$ for particular nondegenerate quartics of simple form.

Formula \eqref{covar1} yields an expression for the morphism $\Phi$, hence ${\mathbb \Phi}$, via $I_2$, $I_3$ and $H$. Namely, for $f\in\XX_2^4$ we obtain
\begin{equation} \label{eqn-quartic}
{\mathbb \Phi}(f) = \frac{1}{2^7 3^3}I_2(f) H(f)(-z_2^*,z_1^*)-\frac{1}{2^4}I_3(f) f(-z_2^*,z_1^*),
\end{equation}
where the right-hand side is viewed as an element of $\PP(\cQ_2^{4*})$. One might hope that formula \eqref{eqn-quartic} provides an extension of ${\mathbb \Phi}$ beyond $\XX_2^4$. However, for $f=z_1^2z_2^2$ the right-hand side of \eqref{eqn-quartic} vanishes, which agrees with the fact, explained in Subsection \ref{S:binary-quartics}, that ${\mathbb \Phi}$ does not have a natural continuation to the orbit $O_1=O(z_1^2z_2^2)$.

In the remainder of this subsection we think of $\Phi$ as a map from $X_2^4$ to $\cQ_2^4$, hence of $\Delta\Phi$ as a self-map of $\cQ_2^4$ (see Remark \ref{olddef}). From formulas \eqref{binary-pullbacks} we then see
\begin{equation} \label{eqn-pullback}
I_2\circ(\Delta \Phi) = \frac{\Delta I_2}{2^8 3^3} \, , \qquad I_3\circ(\Delta \Phi) = -\frac{\Delta^2}{2^{12} 3^6}.
\end{equation}
Furthermore, one can analogously compute
\begin{equation} \label{eqn-involution}
(\Delta \Phi) \circ (\Delta \Phi) =   	
-\frac{I_3 \Delta^2}{ 2^{20} 3^{6}}{\mathbf {id}},
\end{equation}
which verifies the fact, observed in Subsection \ref{S:binary-quartics}, that the rational map $\mathbb{\Phi}$ is an involution. 

We will now derive explicit formulas for the Hessian of the associated form of a binary quartic $f$ and for the associated form of the Hessian of $f$. In effect, we calculate the compositions $H\circ(\Delta\Phi)$ and $(\Delta\Phi)\circ H$. These formulas are interesting in their own right as they provide a better understanding of associated forms and their relation to classical covariants. In particular, they emphasize the difference between $H(f)$ and the associated form of $f$ noted in Remark \ref{invsys}.

By substituting \eqref{covar1} into \eqref{eqn-involution} while appealing to formulas \eqref{eqn-pullback}, we obtain the identity
$$ 
I_2 H\circ(\Delta \Phi) = -\Delta\left[\frac{I_3 \tilde{\mathbf{id}}}{32} + \frac{\Delta \Phi}{2}\right],
$$
where the operation\,\,\, $\tilde{}$\,\,\, is defined as in (\ref{relcovcontrav1}). Applying \eqref{covar1} again, we establish the relation
 \begin{equation}
H\circ(\Delta \Phi) = -\frac{\Delta\tilde H}{2^8 3^3},\label{hofdeltaphi}
\end{equation}
which for every $f\in X_2^4$ leads to the following expression for the Hessian of the associated form ${\mathbf f}$ of $f$:
$$
H({\mathbf f})(z)=-\frac{H(f)(-z_2,z_1)}{2^8 3^3\Delta(f)}. 
$$

Further, by composing each side of formula \eqref{covar1} with $H$ and using the identities
$$
\begin{array}{l}
I_2\circ H = 2^6 3^3 I_2^2,\, I_3\circ H = 2^{10}3^6 I_3^2-2^{9} 3^3 I_2^3,\\
\vspace{-0.3cm}\\
H\circ H = 2^{10}3^6 I_3 {\mathbf {id}} - 2^63^3 I_2 H,
\end{array}
$$
we obtain the relation
\begin{equation}
\begin{array}{l}
(\Delta\Phi)\circ H =2^93^6I_2^2I_3\tilde{\mathbf{id}}-2^63^6I_3^2\tilde H,
 \end{array}\label{deltaphiofh}
\end{equation}
which for $f\in\cQ_2^4$ with $H(f)\in X_2^4$ yields an expression of the associated form of $H(f)$ via classical covariants.

Above we calculated the compositions $H\circ(\Delta\Phi)$ and $(\Delta\Phi)\circ H$ in (\ref{hofdeltaphi}), (\ref{deltaphiofh}) from formulas \eqref{covar1}, \eqref{eqn-pullback}, \eqref{eqn-involution}. On the other hand, identities (\ref{hofdeltaphi}), (\ref{deltaphiofh}) can be also derived analogously to the relations in \eqref{binary-pullbacks}, and one can then obtain \eqref{eqn-involution} from \eqref{covar1}, \eqref{eqn-pullback}, (\ref{hofdeltaphi}).  

\subsection{Binary quintics} \label{S:contra2}
Descriptions of the map $\Phi$ in terms of standard covariants can be also obtained for binary forms of certain degrees higher than 4, but the computations are more involved. Here we briefly sketch our calculations for the case of binary quintics, i.e.~for $n=2$, $d=5$. By Proposition \ref{moreprecise}, in this situation $\Delta \Phi$ is a contravariant of weight 18, degree 6 and class 6.

A generic binary quintic $f \in \cQ_2^5$ is linearly equivalent to a quintic given in the Sylvester canonical form
\begin{equation}
f = a X^5 + b Y^5 + c Z^5,\label{sylvcanform}
\end{equation}
where $X$, $Y$, $Z$ are linear forms satisfying $X+Y+Z=0$ (see, e.g.,~p.~272 in \cite{El}). The algebra of invariants of binary quintics is generated by invariants of degrees 4, 8, 12, 18 with a relation in degree 36, and the algebra of covariants is generated by 23 fundamental covariants (see \cite{Sy}), which we will write as $C_{i,j}$ where $i$ is the degree and $j$ is the order. 

For $f\in\cQ_2^5$ given in the form (\ref{sylvcanform}) the covariants relevant to our calculations are computed as follows:
$$
\begin{array}{l}
C_{4,0}(f)=a^2b^2+b^2c^2+a^2c^2-2abc(a+b+c), \\
\vspace{-0.3cm}\\
C_{8,0}(f)=a^2b^2c^2(ab+ac+bc),\\
\vspace{-0.3cm}\\
C_{5,1}(f)=abc(bcX+acY+abZ), \\
\vspace{-0.3cm}\\
C_{2,2}(f)=abXY+acXZ+bcYZ, \\
\vspace{-0.3cm}\\
C_{3,3}(f)=abcXYZ, \\
\vspace{-0.3cm}\\
C_{4,4}(f)=abc(aX^4+bY^4+cZ^4),\\
\vspace{-0.3cm}\\
C_{1,5}(f)=f =  a X^5 + b Y^5 + c Z^5,\\
\vspace{-0.3cm}\\
\displaystyle C_{2,6}(f)=\frac{H(f)}{400}=abX^3Y^3+bcY^3Z^3+acX^3Z^3.
\end{array}
$$
For instance, the discriminant can be written as
$$
\Delta=C_{4,0}^2-128\,C_{8,0}.
$$

The vector space of covariants of degree 6 and order 6 has dimension 4 and is generated by the products 
$$
\hbox{$C_{4,0}C_{2,6}$, $C_{1,5}C_{5,1}$, $C_{3,3}^2$, $C_{2,2}^3$, $C_{2,2}C_{4,4}$}
$$
satisfying the relation
$$
C_{4,0}C_{2,6} - C_{1,5}C_{5,1} + 9 C_{3,3}^2 - C_{2,2}^3 + 2 C_{2,2}C_{4,4}=0.
$$
One can then explicitly compute
$$
\widehat{\Delta \Phi}=\frac{1}{20}C_{4,0}C_{2,6}-\frac{3}{50}C_{1,5}C_{5,1}+\frac{27}{10}C_{3,3}^2-\frac{1}{10}C_{2,2}^3.
$$

\subsection{Ternary cubics} \label{S:contra3}
Let $n=3$, $d=3$. By Proposition \ref{moreprecise}, in this case $\Delta \Phi$ is a contravariant of weight 10, degree 9 and class 3. Recall that the algebra of invariants of ternary cubics is freely generated by the invariants ${\rm I}_4$, ${\rm I}_6$ defined in Subsection \ref{S:cubics}, and the ring of contravariants is generated over the algebra of invariants by the Pippian ${\rm P}$ of degree $3$ and class $3$, the Quippian ${\rm Q}$ of degree $5$ and class $3$, the Clebsch transfer of the discriminant of degree $4$ and class $6$, and the Hermite contravariant of degree $12$ and class $9$ (see \cite{C}, \cite{MT}). For a ternary cubic of the form \eqref{generaltercubic}, the Pippian and Quippian are calculated as follows:
$$
\hspace{-0.1cm}\begin{array}{l}
{\rm P}(f)(z^*)  = -d(bcz_1^{*3}+acz_2^{*3}+abz_3^{*3})-(abc-4d^3)z_1^*z_2^*z_3^*,\\
\vspace{-0.1cm}\\
{\rm Q}(f)(z^*) = (abc-10d^3)(bcz_1^{*3}+acz_2^{*3}+abz_3^{*3})-\\
\vspace{-0.3cm}\\
\hspace{8cm}6d^2(5abc+4d^3)z_1^*z_2^*z_3^*.
\end{array}
$$
Since any contravariant of degree 9 and class 3 is a linear combination of ${\rm I}_6 {\rm P}$ and ${\rm I}_4 {\rm Q}$, it is easy to compute
\begin{equation}
\Delta \Phi = -\frac{1}{36}{\rm I}_6 {\rm P} - \frac{1}{27}{\rm I}_4 {\rm Q}.\label{contravar3}
\end{equation}
The above expression can be verified directly by applying it to the cubics $c_t$ defined in (\ref{ct}) and using formulas (\ref{bfct}), (\ref{discrtercub}), (\ref{form11}). 

Identity \eqref{contravar3} provides an expression for $\Phi$, therefore ${\mathbb \Phi}$, in terms of ${\rm I}_4$, ${\rm I}_6$, ${\rm P}$ and ${\rm Q}$. Namely, on $\XX_3^3$ we have
\begin{equation}
{\mathbb \Phi} = -\frac{1}{36}{\rm I}_6 {\rm P} - \frac{1}{27}{\rm I}_4 {\rm Q},\label{newexpr1}
\end{equation}
where the right-hand side is regarded as a morphism $\XX_3^3\to\PP(\cQ_3^{3*})$. One might think that formula (\ref{newexpr1}) yields a continuation of ${\mathbb \Phi}$ beyond $\XX_3^3$. However, for $f=z_1z_2z_3$ the right-hand side of (\ref{newexpr1}) is zero, which illustrates the obstruction to extending the morphism ${\mathbb \Phi}$ to the orbit ${\rm O}_1=O(z_1z_2z_3)$ discussed in Subsection \ref{S:cubics}.

Thinking of $\Phi$ as a map from $X_3^3$ to $\cQ_3^3$, hence of $\Delta\Phi$ as a self-map of $\cQ_3^3$, analogously to formula (\ref{eqn-involution}) for binary quartics we obtain
\begin{equation} \label{eqn-involution2}
(\Delta \Phi) \circ (\Delta \Phi) = -\frac{{\rm I}_4^2 \Delta^6}{ 2^{21} 3^{30}}{\mathbf{id}},
\end{equation}
which agrees with the fact, established in Subsection \ref{S:cubics}, that the rational map ${\mathbb \Phi}$ is an involution.
Formula \eqref{eqn-involution2} can be verified either analogously to the relations in (\ref{formauxcubics}) or by using \eqref{contravar3} together with the expressions 
\begin{equation}
\begin{array}{ll}
\displaystyle{\rm I}_4\circ (\Delta \Phi) = -\frac{\Delta^3}{2^{12}3^{12}}, &\displaystyle {\rm I}_6\circ (\Delta \Phi) = -\frac{{\rm I}_6 \Delta^4}{2^{15} 3^{18}},\\
\vspace{-0.1cm}\\
\displaystyle{\rm P}\circ (\Delta \Phi) = \frac{H \Delta^2}{2^{10}3^{12}},&\displaystyle {\rm Q}\circ(\Delta \Phi) = - \frac{H\,{\rm I}_6\Delta^3}{2^{15} 3^{17}} - \frac{{\rm I}_4^2\Delta^3 {\mathbf{id}}}{2^9 3^{15}}.
\end{array}\label{identsauxl}
\end{equation}
The first two equations in (\ref{identsauxl}) follow from (\ref{formauxcubics}), and the remaining two can be derived in a similar way.

\end{document}